\documentclass{ecgd-l}

\usepackage{enumerate,amsrefs,amssymb,amsthm,nicefrac}

\newtheorem{theorem}{Theorem}[section]
\newtheorem{lemma}[theorem]{Lemma}
\newtheorem{proposition}[theorem]{Proposition}
\newtheorem{fact}[theorem]{Fact}
\newtheorem{question}[theorem]{Question}

\theoremstyle{definition}
\newtheorem{definition}[theorem]{Definition}
\newtheorem{example}[theorem]{Example}
\newtheorem{corollary}[theorem]{Corollary}	

\newtheorem*{conjecture}{Conjecture}

\theoremstyle{remark}
\newtheorem{remark}[theorem]{Remark}

\numberwithin{equation}{section}

\newcommand{\eps}{\varepsilon}

\DeclareMathOperator{\dist}{dist}

\DeclareMathOperator{\id}{id}
\DeclareMathOperator{\diam}{diam}

\begin{document}

\title{Topological Conformal Dimension}

\author{Claudio A. DiMarco}
\address{215 Carnegie, Mathematics Department, Syracuse University,
Syracuse, NY 13244-1150}
\email{cdimarco@syr.edu}



\date{\today}

\begin{abstract}
We investigate a quasisymmetrically invariant counterpart of the topological Hausdorff dimension of a metric space.  This invariant, called the \textit{topological conformal dimension}, gives a lower bound on the topological Hausdorff dimension of quasisymmetric images of the space.  We obtain results concerning the behavior of this quantity under products and unions, and compute it for some classical fractals.  The range of possible values of the topological conformal dimension is also considered, and we show that this quantity can be fractional.
\end{abstract}

\maketitle

\section{Introduction}

For a metric space $(X,d),$ the topological dimension can be defined inductively as 
\begin{align*}
\dim_t X=\inf\{c: X\text{ has a basis }\mathcal{U}\text{ such that } \dim_t\partial U\leq c-1~\text{for all}~ U\in\mathcal{U}\}, 
\end{align*}
where $\dim_t\varnothing=-1.$  Since topological dimension is bounded above by Hausdorff dimension, we write $\dim_t X\leq\dim_H X$  and say that $X$ is \textit{fractal} if $\dim_t X<\dim_H X$~\cite{BM}.

Topological dimension is invariant under homeomorphisms, while Hausdorff dimension is bi-Lipschitz invariant.  An intermediate interesting class of maps between the first two is quasisymmetric maps~\cites{JH,TV}.  In order to classify spaces up to quasisymmetric equivalence, it is useful to have a concept of dimension that is quasisymmetrically invariant.  One such example is conformal dimension~\cite{MT}, which we denote $\dim_C X.$  

Conformal dimension is modeled after Hausdorff dimension, which measures the size of a metric space.  Recently, Balka, Buczolich and Elekes~\cite{BBE2} introduced the \textit{topological Hausdorff dimension}, which is a bi-Lipschitz invariant sensitive to connectivity.  Unlike conformal dimension, this dimension is not quasisymmetrically invariant (Theorem~\ref{dimth theorem}).  Indeed, it can be arbitrarily increased by quasisymmetric maps (Fact \ref{effect of snowflake on tH}).  At present, there are no nontrivial lower bounds on topological Hausdorff dimension of quasisymmetric images, unlike for Hausdorff dimension.

In this paper, we introduce a quasisymmetric invariant that delivers such a bound: the \textit{topological conformal dimension}, denoted $\dim_{tC}X.$  Writing $\dim_{tH}X$ for the topological Hausdorff dimension, 
\begin{equation*}
\dim_{tH}f(X)\geq \dim_{tC}X
\end{equation*}
for every quasisymmetric mapping $f$ of $X.$  Our investigation uncovers a parallel between conformal dimension and topological conformal dimension.  The former detects rich families of curves, while the latter detects rich families of surfaces.  By definition,
\begin{equation*}
\dim_{tC}X=\inf\{c:X \text{ has a basis } \mathcal{U} \text{ such that } \dim_C \partial U\leq c-1 \text{ for all } U\in\mathcal{U}\}.
\end{equation*}

We obtain results concerning the behavior of this new quantity under products and unions, along with the range of its possible values.  The topological conformal dimension is computed for some classical fractals, often via comparison to the Hausdorff dimension.  The following theorem, proven in section 4, gives a two-sided estimate for the topological conformal dimension of a space that contains a  ``diffuse" family of surfaces.  This result is similar to that of Pansu; lower bounds for conformal dimension are based on the presence of a diffuse family of curves.  The Assouad dimension of $X$ is denoted $\dim_A X$~\cite{JH}.
\newtheorem*{ftc2}{Theorem \ref{ftc2}}
\begin{ftc2}
Let $(X,d,\mu)$ be a compact, doubling metric measure space with $\dim_A X=D<\infty.$  Suppose that there exists a family $\mathcal{Q}$ of surfaces with the following properties:
\begin{enumerate}[(i)] 
\item[\textnormal{(i)}] there is a family $\{f_{\alpha}:\alpha\in J\}$ of $L-$bi-Lipschitz maps such that $Q_{\alpha}=f_{\alpha}([0,1]^2)$ for each $Q_{\alpha}\in\mathcal{Q}.$
\item[\textnormal{(ii)}] there exists a measure $\nu_0$ on $\mathcal{Q}$, a constant $C_0,$ and $r_0>0$ so that 
\begin{equation*}\label{eqE}
0<\nu_0\{Q_{\alpha}\in\mathcal{Q}:Q_{\alpha}\cap B(x,r)\neq\emptyset\}\leq C_0 r^d
\end{equation*}
\noindent for all $x\in X$ and for all $r\leq r_0.$
\item[\textnormal{(iii)}] The union of edges $\bigcup_{\alpha\in J}f_{\alpha}\left(\partial([0,1]^2)\right)$ is not dense in $X.$
\end{enumerate}
Then $1+\frac{D}{D-d}\leq\dim_{tC} X\leq D .$
\end{ftc2}
Such spaces are shown to exist for every $0<d<1$ with $D=2+d.$  Since $2+\nicefrac{d}{2}\leq \dim_{tC}\leq 2+d$, it follows that the topological conformal dimension attains infinitely many fractional values even though they are not explicitly determined.  In particular, if $0<d<1$ and $D=2+d$, then $2<\dim_{tC}X<3$, which furnishes a collection of compact metric spaces with fractional topological conformal dimension.

A key difference between topological conformal dimension and topological Hausdorff dimension is that conformal dimension may increase under Lipschitz maps.  This makes it difficult to give a lower bound on topological conformal dimension, such as that in Theorem~\ref{ftc2}.

\section{Preliminaries and Basic Properties}

The subscripts of dim indicate the type of dimension.  By convention, every dimension of the empty set is $-1.$  For any set $A,~|A|$ means the cardinality of $A.$  We write $B(x,\eps)$ for the open ball centered at $x$ of radius $\eps.$  We will often refer to a basis of $X,$ by which we mean a basis for the topology on $X$ induced by the metric on $X.$  Our definition of topological conformal dimension is similar to that of the classical topological dimension and the topological Hausdorff dimension.  The \textit{topological Hausdorff dimension} is defined in~\cite{BBE2} as 
\begin{align*}
\dim_{tH}X=\inf\{c:X\text{ has a basis }\mathcal{U}\text{ such that }\dim_H\partial U\leq c-1~\forall U\in\mathcal{U}\}.
\end{align*}
An embedding $f:X\rightarrow Y$ is \textit{quasisymmetric} if there is a homeomorphism $\eta:[0,\infty)\rightarrow [0,\infty)$ so that
\begin{align*}
d_X(x,a)\leq td_X(x,b)~\text{ implies }~d_Y(f(x),f(a))\leq\eta(t)d_Y(f(x),f(b))
\end{align*}
for all triples $a,b,x$ of points in $X,$ and for all $t>0$ ~\cite{JH}.  The conformal dimension is defined via the Hausdorff dimension.  For the latter, recall that the \textit{p-dimensional Hausdorff measure} of $X$ is 
\begin{align*}
\mathcal{H}^p(X)=\lim_{\delta\rightarrow 0}\inf\left\{\sum (\diam E_j)^p: X\subset \bigcup E_j\text{ and }\diam E_j\leq\delta~\forall j\right\}.
\end{align*}
The \textit{Hausdorff dimension} of $X$ is $\dim_H X=\inf\{p:\mathcal{H}^p(X)=0\},$ and the \textit{conformal dimension} of $X$ is 
\begin{align*}
\dim_C X=\inf\left\{\dim_H f(X):f\text{ is quasisymmetric}\right\}.
\end{align*}

\begin{definition}
The \textit{topological conformal dimension} of a metric space $(X,d)$ is
\begin{equation}\label{defc}
\dim_{tC}X=\inf\{c:X \text{ has a basis } \mathcal{U} \text{ such that } \dim_C \partial U\leq c-1 \text{ for all } U\in\mathcal{U}\}
\end{equation}
\end{definition}

We will have occasion to use the following two facts.
\begin{align}
\dim_t X\leq \dim_C X\leq \dim_H X \label{f2}\\ 
\dim_t X\leq\dim_{tC}X\leq\dim_{tH}X\leq\dim_H X \label{f1}
\end{align}

\begin{proof}
Inequality \eqref{f2} is well known~\cite{MT}.  For the first inequality in \eqref{f1}, let $\mathcal{U}$ be a basis of $X$ and $U\in\mathcal{U}.$  By \eqref{f2} we have $\dim_t\partial U\leq\dim_C\partial U,$ so if $\dim_C\partial U\leq c-1$ then $\dim_t\partial U\leq c-1.$  It follows that $\dim_t X\leq\dim_{tC}X.$  Using \eqref{f2}, the same argument gives $\dim_{tC}X\leq\dim_{tH}X.$  The inequality $\dim_{tH}X\leq\dim_H X$ is Theorem~4.4~in~\cite{BBE2}.
\end{proof}
By Fact~4.1~in~\cite{BBE2} we have $\dim_{tH}X=0\Leftrightarrow \dim_t X=0,$ so \eqref{f2} yields
\begin{equation}
\dim_t X=0\Leftrightarrow \dim_{tC}X=0\Leftrightarrow \dim_{tH} X=0.\label{00}
\end{equation}

\begin{proposition}
If $X$ is a metric space, then $\dim_t X\leq \dim_{tC} X\leq \dim_C X.$ 
\end{proposition}

\begin{proof}
 The first inequality is \eqref{f1}.  For the second, let $ \eps>0.$  By definition of $\dim_C X$ there is a quasisymmetric map $f$ with $\dim_H f(X) \leq \dim_C X+\eps.$  Since the topological conformal dimension is a quasisymmetric invariant, we have
\begin{equation}
\dim_{tC}X=\dim_{tC}f(X)\leq\dim_{tH}f(X)\leq\dim_H f(X)\leq\dim_C X+\eps. \qedhere
\end{equation}
\end{proof}

\begin{proposition}
The topological conformal dimension of a countable set is zero, and for open subsets of $\mathbb{R}^d$ and for smooth $d-$dimensional manifolds, the topological conformal dimension is $d.$
\end{proposition}
\begin{proof}
This follows immediately from \eqref{f1}.
\end{proof}

It is clear from the definition that the topological conformal dimension is invariant under quasisymmetric maps.  In particular, it is invariant under bi-Lipschitz maps.  The following examples are along the lines of those in ~\cite{BBE2}.  The next example shows that the topological conformal dimension can increase under Lipschitz maps in general, and Example~4.12~in~\cite{BBE2} gives an example with an injective Lipschitz map.

\begin{example}\label{lip example}
 Let $K\subset [0,1]$ be a Cantor set of positive Lebesgue measure.  Lemma~4.10~in~\cite{BBE2} gives a surjective Lipschitz map $f:K\rightarrow [0,1],~\text{namely}~f(x)=\frac{m((-\infty,x)\cap K)}{m(K)},$ where $m$ is Lebesgue measure.  Since $K$ has a basis of clopen sets, we have $\dim_t K=\dim_{tC}K=0.$  Also $1=\dim_{tC}[0,1]=\dim_{tC}f(K),$ so $\dim_{tC}K<\dim_{tC}f(K).$
\end{example}

\section{Behavior Under Inclusions, Unions, and Products}

As with other types of dimension, tC-dimension is monotone under inclusion.
\begin{proposition} \textnormal{(Monotonicity under inclusion)}
If $X$ is a subset of a metric space $Y$, then $\dim_{tC}X\leq\dim_{tC}Y.$
\end{proposition}
\begin{proof}
If $\mathcal{U}$ is a basis of $Y$ then $\{U\cap X:U\in\mathcal{U}\}$ is a basis in $X$ and $\partial_X(U\cap X)\subset\partial_Y U$ holds for all $U\in\mathcal{U}.$  Note that monotonicity of conformal dimension follows from monotonicity of Hausdorff dimension.  Thus $\dim_C \partial_X(U\cap X)\leq\dim_C\partial_Y U,$ so the result follows.
\end{proof}
  
Conformal dimension is stable under finite disjoint unions of compact sets (\cite{MT}, Proposition~5.2.3).  The same is true of the topological conformal dimension.

\begin{fact}\label{th1}
 If $S_1$ and $S_2$ are disjoint compact subsets of a metric space and if $S=S_1\cup S_2,$ then $\dim_{tC}S=\max\{\dim_{tC}S_1,\dim_{tC}S_2\}.$
\end{fact}

\begin{proof}
 Let $\eps>0.$  We find a basis $\mathcal{U}$ of $S$ and $d\geq 0$ such that $\dim_C \partial U\leq d-1$ for all  $U\in\mathcal{U},$ and such that $d\leq \max\{\dim_{tC}S_1,\dim_{tC}S_2\}+\eps.$  
By definition of the topological conformal dimension, for $i=1,2$ there exist $\mathcal{U}_i,d_i$ such that $\mathcal{U}_i$ is a basis for $S_i,~\dim_C \partial U_i\leq d_i-1$ for all $U_i\in\mathcal{U}_i,$ and $d_i\leq \dim_{tC}S_i+\eps.$

Consider the basis $\mathcal{U}=\mathcal{U}_1\cup\mathcal{U}_2$ of $S$ and let  $d=\max\{d_1,d_2\}.$  For each $U=U_1\cup U_2\in\mathcal{U},~\partial U_1$ and $\partial U_2$ are disjoint compact sets, so $\partial U=\partial U_1\cup \partial U_2$ and Proposition~5.2.3~in~\cite{MT} implies 
\begin{equation}
\begin{split}
\dim_C \partial U=&\max\{\dim_C \partial U_1,\dim_C \partial U_2\}\\ \label{x}
\leq & \max\{d_1-1,d_2-1\}\\
\leq & \max\{\dim_{tC}S_1,\dim_{tC}S_2\}+\eps-1.
\end{split}
\end{equation}
Inequality (\ref{x}) yields $\dim_{tC}S\leq \max\{\dim_{tC}S_1,\dim_{tC}S_2\}$ and  monotonicity gives the opposite inequality.
\end{proof}

\begin{corollary}\label{cor12}
 If $\{X_1,\dots,X_n\}$ is a disjoint family of compact subsets of a metric space, then $\dim_{tC}\bigcup_{i=1}^n X_i=\max\{\dim_{tC} X_i|i=1,\dots,n\}.$
\end{corollary}

Unlike the Hausdorff dimension, the next example shows that finite stability does not hold for non-closed sets.

\begin{example}
Each of $\mathbb{Q}$ and $\mathbb{R}\setminus\mathbb{Q}$ has a basis of clopen sets, so their topological dimensions are zero.  Then \eqref{00} gives $\dim_{tC}\mathbb{Q}=\dim_{tC}(\mathbb{R}\setminus\mathbb{Q})=0.$  By \eqref{f1} we have $\dim_{tC}\mathbb{R}=1$ so that
\begin{align*}
\dim_{tC}\mathbb{R}=1>0=\max\{\dim_{tC}\mathbb{Q},\dim_{tC}(\mathbb{R}\setminus\mathbb{Q})\}.
\end{align*}
\end{example}

\subsection{Dimension of Products}
For two metric spaces $X$ and $Y$ we consider the product space $Z=X\times Y$ with metric
\begin{equation*}
d_Z((x_1,y_1),(x_2,y_2))=\max (d_X(x_1,x_2),d_Y(y_1,y_2)).
\end{equation*}
Under certain favorable conditions, Hausdorff dimension is additive under products~\cite{PM}.  This is not the case for conformal dimension~\cite{MT}.  In section 4 we will see that topological conformal dimension is also not additive under products (Corollary \ref{tc not additive}).  In this section we provide an upper bound on the tC-dimension of the product of two Jordan arcs.  This requires a lemma pertaining to the conformal dimension of the union of two \textit{proper}, \textit{uniformly perfect} subsets of a metric space.

\begin{lemma}\label{arcthm}
Let $(X,d)$ be a metric space with  proper uniformly perfect subsets $A$ and $B$ such that $A\cup B=X.$  Then $\dim_C X=\max\{\dim_C A,\dim_C B\}.$
\end{lemma}

The proof of Lemma \ref{arcthm} relies heavily on Theorem~1~in~\cite{PH} several times, so it is included below.  A metric space $X$ is called \textit{proper} if closed and bounded subsets are compact.  It is called \textit{uniformly perfect} if there is a constant $C\geq 1$ so that for each $x\in X$ and for each $r>0$ the set $B(x,r)\setminus B(x,\frac{r}{C})$ is nonempty whenever the set $X\setminus B(x,r)$ is nonempty~\cite{JH}.  

\begin{theorem}\textnormal{(Haissinsky)}\label{haissinsky}
Let $(X,d_X)$ be a proper metric space containing at least two points and $(Y,d_Y)$ a proper uniformly perfect space.  Suppose there is a quasisymmetric embedding $f:Y\rightarrow X.$  Then there is a metric $\hat{d}$ on $X$ such that 
\begin{enumerate}[(i)]
\item[\textnormal{(i)}] $\id:(X,d_X)\rightarrow (X,\hat{d})$ is quasisymmetric;
\item[\textnormal{(ii)}] $\id:(X\setminus f(Y),d_X)\rightarrow (X\setminus f(Y),\hat{d})$ is locally bi-Lipschitz;
\item[\textnormal{(iii)}] $f:(Y,d_Y)\rightarrow (X,\hat{d})$ is bi-Lipschitz onto its image.
\end{enumerate}
\end{theorem}

\begin{proof}[Proof of Lemma \ref{arcthm}]
Put $c=\dim_C(A,d)$ and $c'=\dim_C(B,d).$  Let $\eps>0.$  By definition of conformal dimension there is a quasisymmetric map $g$ with $\dim_H g(A,d)\leq c+\eps.$  Then $f=g^{-1}:g(A,d)\rightarrow X$ is a quasisymmetric embedding.  Since $A$ and $B$ are proper, it follows that $X$ is also proper.  Indeed, let $E\subset X$ be closed and bounded.  Then $E=(E\cap A)\cup (E\cap B)$ and since $E$ is closed in $X,~E\cap A$ and $E\cap B$ are closed in $A$ and $B,$ respectively.  Since $E$ is bounded, $E\cap A$ and $E\cap B$ are also bounded.  Therefore $E\cap A$ and $E\cap B$ are compact since $A$ and $B$ are proper.  Then $E$ is a union of two compact sets and hence is compact, so $X$ is proper.  By Theorem~\ref{haissinsky} there is a metric $d_1$ on $X$ such that 
\begin{align}
\id\colon&(X,d)\rightarrow (X,d_1)~\text{is quasisymmetric} \label{1}\\
\id\colon&(X\setminus A,d)\rightarrow (X\setminus A,d_1)~\text{is locally bi-Lipschitz}\label{2}\\
f\colon&g(A,d)\rightarrow (A,d_1)~\text{is bi-Lipschitz}\label{3}
\end{align}

Since the restriction of a quasisymmetric map is again quasisymmetric, \eqref{1} implies $\dim_C(B,d_1)=\dim_C(B,d)=c'.$  Therefore there is a quasisymmetric map $r$ such that $\dim_H r(B,d_1)\leq c'+\eps.$  Then $h=r^{-1}:r(B,d_1)\rightarrow (X,d_1)$ is a quasisymmetric embedding with image $(B,d_1).$  Since $r$ is a homeomorphism,  $r(B,d_1)$ is proper and uniformly perfect.  Another application of Theorem~\ref{haissinsky} gives a metric $d_2$ on $X$ such that
\begin{align}
\id:~&(X,d_1)\rightarrow (X,d_2)~\text{is quasisymmetric}\label{4}\\
\id:~&(X\setminus B,d_1)\rightarrow (X\setminus B,d_2)~\text{is locally bi-Lipschitz}\label{5}\\
h:~&r(B,d_1)\rightarrow (B,d_2)~\text{is bi-Lipschitz}\label{6}
\end{align}

Locally bi-Lipschitz maps preserve Hausdorff dimension, so \eqref{5} and \eqref{3} yield
\begin{equation}\label{6'}
\begin{split}
\dim_H(X\setminus B,d_2)&=\dim_H(X\setminus B,d_1)\\
&\leq\dim_H(A,d_1)\\
&=\dim_H g(A,d).
\end{split}
\end{equation}
Also \eqref{6} shows that $\dim_H(B,d_2)=\dim_H r(B,d_1),$ so by \eqref{6'} we have
\begin{equation}
\begin{split}
\dim_H(X,d_2)=&\max\{\dim_H(X\setminus B,d_2),\dim_H(B,d_2)\}\label{7}\\
\leq& \max\{\dim_H g(A,d),\dim_H r(B,d_1)\}\\
\leq& \max\{c+\eps,c'+\eps\}\\
=&\max\{c,c'\}+\eps.
\end{split}
\end{equation}

Finally \eqref{4}, \eqref{7} and monotonicity give $\dim_C (X,d_1)=\max\{c,c'\}$ so that (\ref{1}) implies $\dim_C (X,d)=\max\{c,c'\}.$
\end{proof}

\begin{theorem}
If $\Gamma$ and $\Lambda$ are Jordan arcs then 
\[
2\leq\dim_{tC}(\Gamma\times\Lambda)\leq \max\{\dim_C\Gamma,\dim_C\Lambda\}+1.
\]
\end{theorem}

\begin{proof}
The first inequality is due to the fact that $2=\dim_t (\Gamma\times \Lambda)\leq \dim_{tC}(\Gamma\times\Lambda).$  For the second inequality, let $\mathcal{U}=\{\gamma((c,d))\times \lambda((a,b)):a,b,c,d\in[0,1]\}$ where $\gamma,~\lambda$ are the homeomorphisms parametrizing $\Gamma,~\Lambda$ respectively.  Let $A\times B\in\mathcal{U}$ where $A=\gamma(c,d),~B=\lambda(c,d).$ Since $A$ and $B$ are open, the interior of $A\times B$ is again $A\times B,$ so we may write $\partial (A\times B)$ as
\begin{align}
\partial(A\times B)=&\overline{A\times B}\setminus (A\times B) \notag \\
=&(\partial A\times B)\cup (A\times\partial B)\cup (\partial A\times\partial B)  \label{8} \\
=&\left((\overline{A}\times\{\lambda(c)\})\cup (\{\gamma(b)\}\times \overline{B})\right)\cup\left( \overline{A}\times\{\lambda(d)\})\cup (\{\gamma(a)\}\times\overline{B})\right). \notag
\end{align}

Let $S=(\overline{A}\times\{\lambda(c)\})\cup (\{\gamma(b)\}\times \overline{B})$ and $T=(\overline{A}\times\{\lambda(d)\})\cup (\{\gamma(a)\}\times\overline{B}).$  Since each of $S$ and $T$ is a union of two connected sets, Lemma \ref{arcthm} gives 
\begin{equation}
\begin{split}
\dim_C S=&\max\{\dim_C (\overline{A}\times\{\lambda(c)\}),\dim_C  (\{\gamma(b)\}\times \overline{B})\} \\ =& \max\{\dim_C \overline{A},\dim_C \overline{B}\}\label{9}\\
\dim_C T=&\max\{\dim_C (\overline{A}\times\{\lambda(d)\}),\dim_C (\{\gamma(a)\}\times\overline{B})\} \\ = &\max\{\dim_C \overline{A},\dim_C \overline{B}\}
\end{split}
\end{equation}
Since $S$ and $T$ are connected, Lemma \ref{arcthm}, \eqref{8} and \eqref{9} yield
\begin{align*}
\dim_C \partial(A\times B)=&\max\{\dim_C S,\dim_C T\}\\
=&\max\{\dim_C \overline{A},\dim_C \overline{B}\}\\
\leq&\max\{\dim_C \Gamma,\dim_C \Lambda\}.
\end{align*}
The conclusion now follows from the definition of the topological conformal dimension.
\end{proof}

We conclude this section with a simple observation that classifies product spaces consisting of two factors; one factor is the unit interval, the other has conformal dimension zero.

\begin{fact}\label{easy product proposition}
If $\dim_C X=0$ then $\dim_{tC}(X\times [0,1])=1.$
\end{fact}

\begin{proof}
Since $X\times [0,1]$ contains a line segment, $\dim_{tC} (X\times [0,1])\geq 1.$  On the other hand, $X$ has a basis $\mathcal{U}$ of clopen sets since $\dim_t X\leq\dim_C X=0.$  Then for every $U\in\mathcal{U}$ and $0\leq a\leq b\leq 1$ we have
\begin{align*} 
\dim_C \partial(U\times (a,b))&=\dim_C (U\times \partial(a,b))\\
&=\dim_C (U\times\{a,b\})\\
&=\dim_C U\\
&\leq\dim_C X=0.\qedhere
\end{align*}
\end{proof}

\section{Range of Dimension} 

Taking products of Cantor sets, one sees that the Hausdorff dimension attains all values in $[0,\infty]$.  Topological Hausdorff dimension attains all values in $[1,\infty]$ by Theorem~4.24~in~\cite{BBE2}.  It was shown in ~\cite{JT} that conformal dimension attains all values in $[1,\infty],$ and later in \cite{LK} that it does not attain any value in $(0,1)$.  Possible values of the topological conformal dimension are restricted to $\{-1,0,1\}\cup[2,\infty].$  

\begin{fact}
The topological conformal dimension cannot take any value in $(1,2).$
\end{fact}
\begin{proof}
If $X=\varnothing$ then $\dim_{tC}X =-1.$  Let $X$ be a nonempty metric space with $\dim_{tC}X=c$ and suppose $1<c<2.$  By definition of the topological conformal dimension there is a number $d$ and a basis $\mathcal{U}$ for $X$ such that $1<c<d<2$ and $\dim_C\partial U\leq d-1$ for all $U\in\mathcal{U}.$  Then \cite{LK} implies $\dim_C \partial U=0$ for all $U\in\mathcal{U}$ so that $\dim_{tC}X\leq 1,$ a contradiction.
\end{proof}

For tC-dimension, all integer values in $[2,\infty]$ are realized by Euclidean spaces.  It is more involved to obtain examples that attain non-integer values.  Theorem~\ref{ftc2} gives one method for accomplishing this task.  In particular, it shows that there are compact subsets of $\mathbb[0,1]^3$ with fractional topological conformal dimension. 

\begin{definition}A metric space $X$ is \textit{Ahlfors $d$-regular} if it admits a Borel regular measure $\mu$ such that 
\begin{equation}
\frac{1}{K}r^d\leq\mu\left(\overline{B}(x,r)\right)\leq Kr^d
\end{equation}
for some constant $K\geq 1$ for all balls ${B}(x,r)$ of radius $0<r<2\diam X.$
\end{definition}

The primary tool we use to prove Theorem~\ref{ftc2} is Proposition~4.1.3~in~\cite{MT}, so it is included here for sake of completeness.

\begin{proposition}\label{MT prop}\textnormal{(Pansu)}
Let $(Z,d,\mu)$ be a compact, doubling metric measure space, and let $1<p<\infty$ and $p'=\frac{p}{p-1}.$ Suppose that there exists a family $\mathcal{E}$ of connected sets in $Z$ and a probability measure $\nu$ on $\mathcal{E}$ with the following properties:
\begin{enumerate}[(i)]
\item[\textnormal{(i)}] there exists $c>0$ so that $\diam E\geq c$ for all $E\in\mathcal{E},$ and 
\item[\textnormal{(ii)}] there exists $C$ and $r_0>0$ so that 
\begin{equation}\label{eq10}
\nu\{E\in\mathcal{E}:E\cap B(x,r)\neq\varnothing\}\leq C\mu(B(x,r))^{\nicefrac{1}{p'}}
\end{equation}
for all balls $x\in Z$ and for all $r\leq r_0.$
\end{enumerate}
Then $\dim_C Z\geq p.$
\end{proposition}

\begin{lemma}\label{fractaltclemma}
If $A\subset \mathbb{C}$ is open, connected, and bounded, then there is a connected set $E\subset\partial A$ such that $\diam E=\diam A.$
\end{lemma}

\begin{proof}
Let $\{S_{\alpha}|\alpha\in J\}$ be the set of connected components of $\overline{A}^c$ and let $\beta\in J$ be such that $S_{\beta}$ is unbounded.  We will show that $E=\partial S_{\beta}$ satisfies the conclusion of the lemma.

The Phragm\'en-Brouwer theorem (Theorem VI.2.1~in~\cite{GW}) shows that $\partial S_{\beta}$ is connected.  Note that $\partial S_{\beta}\subset \partial A.$  Indeed, if $w\in \partial S_{\beta}\setminus \partial A$ then either $w\in A$ or $w\in \overline{A}^c.$  If $w\in \overline{A}^c$ then $w\in S_{\gamma}$ for some $\gamma\neq \beta.$  Since $\overline{S_{\beta}}$ is connected and $\partial S_{\beta}\subset \overline{S_{\beta}},$ we have $S_{\beta}\subset S_{\gamma},$ a contradiction.  If $w\in A$ then $S_{\beta}\subset\overline{S_{\beta}}\subset A$, a contradiction.  

For $z\in\mathbb{C}$ let $\text{Arg}(z)$ be the principal argument of $z.$  Without loss of generality, let $0\in A$ and choose an argument $-\pi<\phi\leq\pi$ and a point $z\in\mathbb{C}$ with $|z|>\diam\overline{A}$ and $\text{Arg}(z)=\phi.$  Connect $0$ to $z$ by a line segment; call it $L.$  The function $w\mapsto |w|$ is continuous on the compact set $L\cap\overline A,$ so it attains a maximum on $L\cap\overline{A},$ say at $z_{\phi}.$  

Put $F=\{z_{\phi}:-\pi<\phi\leq\pi\}.$  We claim that $\diam F\geq\diam A$.  To see this, choose $a,b\in\partial A$ such that $|a-b|=\diam A.$  We will show that $a,b\in F.$  Since $0\in A$ we have $|a-b|>|a|$ and $|a-b|>|b|.$  Note that the function $f(t)=|ta-b|$ is strictly convex on $[0,\infty)$ and $f(0)=|b|<|a-b|=f(1).$  Then $f$ is strictly increasing on $[1,\infty)$ so that $ta\notin A$ for all $t>1.$  It follows that $a\in F$ and similarly $b\in F.$  Clearly  $F\subset(\partial A)\cap(\partial S_{\beta})$.  Therefore $\diam \partial S_{\beta}\geq \diam F\geq\diam A.$  Since $\partial S_{\beta}\subset \overline{A}$ we see that $\diam \partial S_{\beta}\leq\diam\overline{A}=\diam A,$ so $\diam E=\diam A$ with $E=\partial S_{\beta}.$
\end{proof}

The following theorem is the main result of this section.  Recall that $f:X\rightarrow Y$ is $L-$bi-Lipschitz if both $f$ and $f^{-1}$ are $L-$Lipschitz.  We write $\dim_A X$ for the Assouad dimension of $X,$ see \cite{JH},\cite{MT}.
\begin{theorem}\label{ftc2}
Let $(X,d,\mu)$ be a compact, doubling metric measure space with $\dim_A X=D<\infty.$  Suppose that there exists a family $\mathcal{Q}$ of surfaces with the following properties:
\begin{enumerate}[(i)] 
\item[\textnormal{(i)}] There is a family $\{f_{\alpha}:\alpha\in J\}$ of $L-$bi-Lipschitz maps such that $Q_{\alpha}=f_{\alpha}([0,1]^2)$ for each $Q_{\alpha}\in\mathcal{Q}.$
\item[\textnormal{(ii)}] There exists a measure $\nu_0$ on $\mathcal{Q}$, a constant $C_0,$ and $r_0>0$ so that 
\begin{equation}\label{eqE}
0<\nu_0\{Q_{\alpha}\in\mathcal{Q}:Q_{\alpha}\cap B(x,r)\neq\emptyset\}\leq C_0 r^d
\end{equation}
\noindent for all $x\in X$ and for all $r\leq r_0.$
\item[\textnormal{(iii)}] The union of edges $\bigcup_{\alpha\in J}f_{\alpha}\left(\partial([0,1]^2)\right)$ is not dense in $X.$
\end{enumerate}
Then $\dim_{tC} X\geq 1+\frac{D}{D-d}.$
\end{theorem}

\begin{corollary}\label{ftc3}
Let $(C,\nu_C)$ be a compact Ahlfors $d-$regular metric measure space.  If $X=C\times[0,1]^2$ then $2+\nicefrac{d}{2}\leq\dim_{tC}X\leq 2+d.$ 
\end{corollary}

\begin{proof}[Proof of Theorem \ref{ftc2}]
Let $\mathcal{U}$ be a basis for $X$.  For each $\alpha$, call $f_{\alpha}(\partial_{\mathbb{C}}([0,1]^2))$ the \textit{edge} of $Q_{\alpha}.$  In this proof, $\partial G$ means the boundary of $G$ in $Q_{\alpha}.$  Boundaries in other spaces are indicated by subscripts.  Assumption (iii) shows that there is an open set $U'\subset X$ that does not meet the edge of $Q_{\alpha}$ for any $\alpha.$  Choose a nonempty $U\in\mathcal{U}$ such that $\overline U\subset U'.$  We will show that $\dim_C \partial_X U\geq \frac{D}{D-d}$ using Proposition \ref{MT prop}. 

Choose a compact set $K\subset U$ with nonempty interior, and for each $\alpha$ with $K\cap Q_{\alpha}\neq\varnothing$ choose $z_{\alpha}\in K\cap Q_{\alpha}.$  Put $J_K=\{\alpha\in J:K\cap Q_{\alpha}\neq\varnothing\}.$  Let $B_{\alpha}$ be the connected component of $U\cap Q_{\alpha}$ containing $z_{\alpha}.$  Consider $A_{\alpha}=f_{\alpha}^{-1}(B_{\alpha})$.  Note that $A_{\alpha}$ is open, connected, and does not intersect $\partial_{\mathbb{C}}([0,1]^2),$ so Corollary \ref{fractaltclemma} gives a connected set $E_{\alpha}\subset\partial_{\mathbb{C}} A_{\alpha}$ with $\diam E_{\alpha}=\diam A_{\alpha}.$  Since $f_{\alpha}$ is a homeomorphism, $f_{\alpha}(E_{\alpha})$ is connected.  

To satisfy the first part of condition (i) of Proposition \ref{MT prop}, define  $\mathcal{E}=\{f_{\alpha}(E_{\alpha}):\alpha\in J_K\}.$  We need to show that $f_{\alpha}(E_{\alpha})\subset\partial_X U.$  Since $B_{\alpha}$ is a component of $U\cap Q_{\alpha}$ we have $\partial B_{\alpha}\subset \partial(U\cap Q_{\alpha}).$  Then
\begin{equation}
\begin{split}
f_{\alpha}(E_{\alpha})\subset f_{\alpha}(\partial_{\mathbb{C}} A_{\alpha})&=\partial f_{\alpha}(A_{\alpha})=\partial B_{\alpha}\subset \partial(U\cap Q_{\alpha})\subset\partial_X U.
\end{split}
\end{equation}

To satisfy \eqref{eq10} we must show that there is $c>0$ such that $\diam f_{\alpha}(E_{\alpha})\geq c$ for all $\alpha\in J_K$.  We have
\begin{equation}
\begin{split}\label{eq13}
\diam f_{\alpha}(E_{\alpha})&\geq \frac{1}{L}\diam E_{\alpha}\\
&=\frac{1}{L}\diam A_{\alpha}\\
&\geq \frac{1}{L^2}\diam B_{\alpha}\\
\end{split}
\end{equation}
for all $\alpha\in J_K,$ so it suffices to show that there is $c'>0$ with $\diam(B_{\alpha})\geq c'$ for all $\alpha\in J_K.$  Since $\overline{B_{\alpha}}$ is connected and meets both $K$ and $\partial_X U,$ we see that $\diam B_{\alpha}\geq\dist(K,\partial_XU)=c'>0$.  Then \eqref{eq13} yields 
\begin{equation}
\begin{split}\label{eq14}
\diam f_{\alpha}(E_{\alpha})\geq \frac{1}{L^2}\diam (B_\alpha)\geq \frac{c'}{L^2}
\end{split}
\end{equation}
for all $\alpha\in J_K,$ so condition (i) of Proposition \ref{MT prop} is satisfied with $c=\nicefrac{c'}{L^2}$.

To establish \eqref{eq10} we let $s>D$ and show that \eqref{eq10} holds with $p=\frac{s}{s-d}$ and $p'=\frac{s}{d}.$  Put $r_0=\diam\partial_X U.$  We need to use an existence theorem for doubling measures proved by Vol'berg and Konyagin, and extended by Luukainen and Saksman~\cite{LS}.  Since $s>D=\dim_A X\geq\dim_A\partial_X U,$ the Corollary to Theorem 1 in \cite{LS} together with Theorem 13.5 in \cite{JH} show that $\partial_X U$ carries a $(t_0,s)-$homogeneous measure for some $t_0.$  So there is a doubling measure $\mu_0$ on $\partial_X U$ satisfying 
\begin{equation}\label{eq15}
\mu_0(B(x,\lambda r))\leq t_0\lambda^s\mu_0(B(x,r))
\end{equation}
for all $x\in \partial_X U,~r>0,$ and $\lambda\geq 1.$  Fix such a ball $B(x,r).$  With $\lambda=\frac{\diam\partial_X U}{r},$ inequality \eqref{eq15} gives
\begin{equation}\label{eq16}
\mu_0(B(x,r))\geq\frac{\mu_0(\partial_X U)}{t_0(\diam\partial_X U)^s}r^s=Mr^s
\end{equation}
so that
\begin{equation}\label{eq17}
\mu_0(B(x,r))^{\nicefrac{1}{p'}}\geq M^{\nicefrac{1}{p'}}r^d.
\end{equation}

To finish establishing \eqref{eq10} we need to define a measure $\nu$ on $\mathcal{E}.$  Let $F_{\alpha}=f_{\alpha}(E_{\alpha}),$ and for $F\subset \mathcal{E}$ define 
\begin{equation}\label{eq18}
\nu(F)=K_0\nu_0\{Q_{\alpha}:F_{\alpha}\in F\},
\end{equation}
where $K_0$ is such that $\nu(\mathcal{E})=1.$  Note that $F_{\alpha}\cap B(x,r)\neq\varnothing$ implies $Q_{\alpha}\cap B(x,r)\neq\varnothing.$  Coupled with \eqref{eq18} and \eqref{eqE}, this yields
\begin{equation}\label{eq19}
\begin{split}
\nu\{F_{\alpha}:F_{\alpha}\cap B(x,r)\neq\varnothing\}&\leq\nu\{F_{\alpha}:Q_{\alpha}\cap B(x,r)\neq\varnothing\}\\
&=K_0\nu_0\{Q_{\alpha}:Q_{\alpha}\cap B(x,r)\neq\varnothing\}\\
&\leq K_0 C_0 r^d.
\end{split}
\end{equation}
Now \eqref{eq19} and \eqref{eq17} give
\begin{equation}\label{eq20}
\begin{split}
\nu\{F_{\alpha}:F_{\alpha}\cap B(x,r)\neq\varnothing\}&\leq K_0 C_0 r^d\\
&\leq \frac{K_0 C_0}{M^{\nicefrac{1}{p'}}}\mu_0(B(x,r))^{\nicefrac{1}{p'}}\\
&=C\mu_0(B(x,r))^{\nicefrac{1}{p'}}.
\end{split}
\end{equation}

Since the choice of $B(x,r)$ was arbitrary, \eqref{eq20} holds for all such balls, so we have \eqref{eq10}.  Therefore $\dim_C \partial_X U\geq p=\frac{s}{s-d}$ for all $s>D.$  Letting $s\rightarrow D$ gives $\dim_C \partial_X U\geq \frac{D}{D-d}$ and hence $\dim_{tC}X\geq 1+\frac{D}{D-d}.$
\end{proof}

\begin{proof}[Proof of Corollary \ref{ftc3}]
The family of surfaces $\mathcal{Q}=\{\{x\}\times [0,1]^2: x\in C\}$ satifies conditions (i) and (iii) of Theorem \ref{ftc2}.  It remains to show that condition (ii) holds.  
 
Let $d=\dim_H C$ and let $\pi_1:X\rightarrow C$ be the projection of $X$ onto $C.$  For $F\subset \mathcal{Q}$ define
\begin{equation}\label{eq21}
\nu_0(F)=\nu_C\left(\bigcup_{Q\in F} \pi_1(Q)\right).
\end{equation}
By Ahlfors regularity, there is a constant $N$ such that $\nu_C(\pi_1(B(x,r)))\leq Nr^d$ for all $z\in X$ and $0<r\leq\diam X.$  It now follows from \eqref{eq21} that 
\begin{equation}\label{eq22}
\begin{split}
\nu_0\{Q\in\mathcal{Q}: Q\cap B(z,r)\neq\varnothing\}&=\nu_C\left(\bigcup_{E\cap B(z,r)\neq\varnothing}\pi_1(Q)\right)\\
&\leq \nu_C(\pi_1(B(z,r)))\\
&\leq Nr^d,
\end{split}
\end{equation}
which is condition (ii).  Since $X$ is Ahlfors $D-$regular with $D=2+d$ we have
\begin{equation}\label{eq23}
\dim_A(C\times [0,1]^2)=\dim_H (C\times [0,1]^2)=2+d=D.
\end{equation}
Equation \eqref{eq23} and Theorem \ref{ftc2} yield
\begin{equation*}
2+\nicefrac{d}{2}\leq\dim_{tC}X\leq 2+d. \qedhere
\end{equation*}
\end{proof}

\begin{corollary}\label{tc not additive}
Topological conformal dimension is not additive under products.  Moreover, this quantity can be fractional.
\end{corollary}
\begin{proof}
Letting $C$ be the middle-thirds Cantor set, the hypotheses of Corollary \ref{ftc3} are satisfied with $d=\frac{\ln(2)}{\ln(3)}$.  Since $C$ is totally disconnected, $\dim_{tC}C=0,$ so 
\begin{equation}\label{eq24}
\dim_{tC} C+\dim_{tC} ([0,1]^2)=2< 2+\nicefrac{d}{2}\leq\dim_{tC}(C\times [0,1]^2).
\end{equation}
Also $\dim_{tC} (C\times [0,1]^2)\leq 2+d<3,$ so \eqref{eq24} shows that  $\dim_{tC}(C\times [0,1]^2)$ is not an integer.
\end{proof}

Theorem \ref{ftc2} provides hope for an affirmative answer to the following conjecture.

\begin{conjecture}
For every $d\in[2,\infty]$ there is a metric space $X$ with $\dim_{tC}X=d.$
\end{conjecture}

\section{Comparison with Other Dimensions}

In this section, examples are given to demonstrate that topological conformal dimension is different from topological dimension, conformal dimension, topological Hausdorff dimension, and Hausdorff dimension.  It is also shown that conformal dimension and tH-dimension are not comparable.  We compute the tC-dimensions of some classical fractals.  In particular, this type of dimension classifies the Sierpinski carpet as dimension 1, unlike the Hausdorff, topological Hausdorff, and conformal dimensions.  

Due to the additivity of Hausdorff dimension under products and to the fact that  topological dimension only assumes integer values, the following fact is readily seen by considering Corollary \ref{tc not additive}.

\begin{fact}
Topological conformal dimension is different from topological dimension and Hausdorff dimension.
\end{fact}

Write $SG$ for the Sierpinski gasket and $SC$ for the Sierpinski carpet.  Laakso proved that $\dim_C SG=1$~\cite{TW}, while the value $\dim_C SC$ remains unknown~\cite{MT}.

\begin{example}\label{tc of carpet}
$\dim_{tC}SG=\dim_{tC}SC=1.$
\end{example}
\begin{proof}
Since $SG$ contains a line segment, $\dim_{tC}SG\geq 1,$ and we have $\dim_{tC}SG\leq\dim_H SG=\frac{\ln(3)}{\ln(2)}<2.$  Since $\dim_{tC}SG\notin (1,2),$ it follows that $\dim_{tC}SG=1.$

Similarly, since $SC$ contains a line segment, $\dim_{tC}SC\geq 1$.  On the other hand, $\dim_{tC}SC\leq\dim_H SC=\frac{\ln(8)}{\ln(3)}<2$ so that $\dim_{tC}SC=1.$
\end{proof}

In light of Example \ref{tc of carpet}, we observe the following fact.  

\begin{fact}
Topological conformal dimension is different from conformal dimension and topological Hausdorff dimension.
\end{fact}
\begin{proof}
If $C$ is the middle-thirds Cantor set, monotonicity of conformal dimension implies $\dim_C SC\geq \dim_C (C\times [0,1]).$  By Proposition 4.1.11 in \cite{MT},  $\dim_C (C\times [0,1])=1+\frac{\ln(2)}{\ln(3)}$, so $\dim_C SC\geq 1+\frac{\ln(2)}{\ln(3)}.$  By Theorem~5.4~in~\cite{BBE2}, $\dim_{tH}SC=1+\frac{\ln(2)}{\ln(3)}.$
\end{proof}

\begin{fact}
Conformal dimension and topological Hausdorff dimension are not comparable.
\end{fact}
\begin{proof}
We produce spaces $X$ and $Y$ such that $\dim_C X<\dim_{tH}X$ and $\dim_C Y>\dim_{tH}Y.$  Let $0<\alpha<1$ and let $\left([0,1]^n\right)^{\alpha}$ be the metric space that results from snowflaking $[0,1]^n$, as in Fact~\ref{effect of snowflake on tH} below.  By Fact~\ref{effect of snowflake on tH},
\begin{equation*}
\dim_{tH}\left(([0,1]^n)^{\alpha}\right)=\frac{n-1}{\alpha}+1>n=\dim_C\left(([0,1]^n)^{\alpha}\right).
\end{equation*}
On the other hand, if $K\subset\mathbb{R}$ is a uniformly perfect Cantor set with $\dim_HK=1,$ then $\dim_CK=1>\dim_{tH}K=0$ by Corollary 3.3 in \cite{H}.
\end{proof}

\begin{proposition}\label{arcprop}
 If $K$ is homeomorphic to $[0,1],$ then $\dim_{tC}K=1.$
\end{proposition}
\begin{remark}In particular, the topological conformal dimension of the von Koch snowflake is $1.$
\end{remark}
\begin{proof}[Proof of Proposition \ref{arcprop}]
Inequality (\ref{f2}) implies $\dim_{tC}K\geq 1.$  On the other hand, if $K$ is homeomorphic to $[0,1]$ then the usual basis of $[0,1]$ gives a basis $\mathcal{U}$ for $K$ such that $|\partial U|= 2$ for all $U\in\mathcal{U}.$  Then $\dim_C \partial U=0$ for all $U\in\mathcal{U}$ and hence $\dim_{tC}K\leq 1.$
\end{proof}

There are some well known fractals for which the tC-dimension remains unknown, such as Rickman's rug (defined in~\cite{MT}) and the Heisenberg group.  We can, however, compute both the tC and tH-dimensions of the Menger sponge (defined in \cite{MT}, section 3.5).  The following theorem is useful for computing the tH-dimension of the sponge.

\begin{theorem}\label{dimth product}
If $X$ is nonempty, locally compact, Ahlfors $d$-regular, and totally disconnected, and if  $Y$ is any separable metric space with $\dim_{tH}Y\geq 1$, then $\dim_{tH}(X\times Y)=\dim_H X+\dim_{tH}Y.$
\end{theorem}

\begin{proof}
If $\dim_{tH}Y=\infty$ then the statement is trivial, so assume $\dim_{tH}Y<\infty.$  We will first show 
\begin{equation}\label{tH1}
\dim_{tH}(X\times Y)\leq \dim_H X+\dim_{tH}Y.
\end{equation}
Since $X$ is locally compact, $\dim_t X=0$ by Theorem~29.7~in~\cite{SW}, so X has a basis $\mathcal{U}$ consisting of clopen sets.  Let $\eps>0$ and choose a basis $\mathcal{V}$ of $Y$ such that $\dim_H\partial V\leq\dim_{tH}V-1+\eps$ for all $V\in\mathcal{V}.$  For $U\times V\in \mathcal{U}\times\mathcal{V}$ we have $\partial (U\times V)=U\times\partial V.$ Since $X$ is Ahlfors regular, Theorem~5.7~in~\cite{PM} implies that its upper Minkowski dimension $\overline{\dim}_M X$ is equal to $\dim_H X$. Corollary~8.10~in~\cite{PM} gives
\begin{equation}
\begin{split}
\dim_H(U\times\partial V)\leq \dim_H(X\times\partial V)&=\dim_H X+\dim_H\partial V\\
&\leq \dim_H X+\dim_{tH}Y-1+\eps. \label{10}
\end{split}
\end{equation}
Since $\eps>0$ was arbitrary \eqref{10} implies \eqref{tH1}. 

For the reverse inequality, we show that for any $S\subset (X\times Y)$ where $(X\times Y)\setminus S$ is totally disconnected, 
\begin{equation}
\dim_H S\geq \dim_H X+\dim_{tH}Y-1.\label{13}
\end{equation}
Once \eqref{13} is established for all such $S$, Theorem 3.6 in \cite{BBE2} will give the result.  For example, if $X$ is the middle-thirds Cantor set and if $Y=[0,1]^2$, one such subset is $S=X\times (Y\setminus (C\times C)).$  To this end, suppose there is $S$ such that (\ref{13}) fails.  Choose $\beta$ such that 
\begin{equation}
\dim_{H}(S)<\beta<\dim_H X+\dim_{tH} Y-1  \label{11}
\end{equation}

Note that $X$ is proper since it is complete and Ahlfors regular.  Put $\dim_H X=d.$  By virtue of the coarea inequality (Theorem~2.10.25~in~\cite{HF}) 
\begin{equation}\label{coarea}
\int_X^* \mathcal{H}^{\beta-d}(S\cap(\{x\}\times Y))~d\mathcal{H}^d(x)\leq C \mathcal{H}^{\beta}(S).
\end{equation}
Then $\mathcal{H}^d(S)=0$ since $\beta>\dim_H S,$ so there exists $x\in X$ such that 
\begin{equation}\label{12}
\dim_H (S\cap (\{x\}\times Y))\leq \beta-\dim_H X.
\end{equation}
On the other hand, since $Y$ is separable,
\begin{equation}\label{14}
\dim_H(S\cap(\{x\}\times Y))\geq\dim_{tH}Y-1
\end{equation}
by Theorem 3.6 in \cite{BBE2}.  Inequalities \eqref{12} and \eqref{14} yield
\begin{align*}
\dim_{tH}Y-1\leq\beta-\dim_H X,
\end{align*}
which contradicts \eqref{11}.
\end{proof}

Let $M$ denote the Menger sponge.  We compute $\dim_{tH}M$ followed by $\dim_{tC}M.$

\begin{example}\label{th of sponge}
$\dim_{tH}M=1+\frac{\ln(4)}{\ln(3)}.$
\end{example}
\begin{proof}
Consider the sets of intervals 
\begin{align*}
W'&=\{[0,b):0<b<1;~b~\text{is a dyadic rational}\},\\
W''&=\{(a,1]:0<a<1;~a~\text{is a dyadic rational}\},\\
W'''&=\{(a,b):0< a<b<1;~a~\text{and}~b~\text{are dyadic rationals}\}.
\end{align*}
Put $W=W'\cup W''\cup W'''$ and let $\mathcal{U}=M\cap(W\times W\times W).$  Note that $\mathcal{U}$ is a basis of $M$ such that for all $U\in \mathcal{U},~ \partial U$ is a union of at most six sides of a cube in $M.$  Each side of $\partial U$ is a finite union of sets geometrically similar to $C\times C,$ where $C$ is the middle-thirds Cantor set.  Then $\dim_H\partial U=\frac{\ln(4)}{\ln(3)}$ so that $\dim_{tH}M\leq 1+\frac{\ln(4)}{\ln(3)}.$

On the other hand, $M$ contains $C\times SC$, so $\dim_{tH}M\geq\dim_{tH}(C\times SC)$ by monotonicity of the topological Hausdorff dimension.  Note that Theorem~5.4~in~\cite{BBE2} gives $\dim_{tH}SC=1+\frac{\ln(2)}{\ln(3)}$.  Since $C\times SC$ satisfies the hypotheses of Theorem~\ref{dimth product}, 
\begin{align*}
\dim_{tH}(C\times SC)=\dim_H C+\dim_{tH}SC=1+\frac{\ln(4)}{\ln(3)}.
\end{align*}
Thus $\dim_{tH}M\geq 1+\frac{\ln(4)}{\ln(3)}.$
\end{proof}

\begin{example}
$\dim_{tC}M=1.$
\end{example}
\begin{proof}
Let $\mathcal{U}$ be as in Example \ref{th of sponge}.  A similar argument shows that $\dim_{tC}M\leq 1+\dim_C (C\times C).$  Since $C\times C$ is uniformly disconnected, $\dim_C (C\times C)=0$~(\cite{MT}, section 1.3], so $\dim_{tC}M\leq 1.$  By monotonicity $\dim_{tC}M\geq 1.$
\end{proof}

\section{Quasisymmetric Distortion of Topological Hausdorff Dimension}

Topological Hausdorff dimension is not quasisymmetrically invariant.  For example, it increases under the snowflake transformation: 
\begin{fact}\label{effect of snowflake on tH}
Let $(X,d)$ be a metric space and let $0<\alpha<1.$  If $X^{\alpha}=(X,d^{\alpha})$ is the snowflaked metric space defined by the transformation $d(x,y)\mapsto d(x,y)^{\alpha}$, then $\dim_{tH} (X^{\alpha})=\frac{1}{\alpha}(\dim_{tH}X-1)+1.$
\end{fact}
\begin{proof}
For any metric space $Z,$ snowflaking has the following effect on Hasdorff dimension: $\dim_H (Z^{\alpha})=\frac{1}{\alpha}\dim_H Z$ (\cite{MT}, Corollary 1.4.18).  Applying this fact to open subsets of $X^{\alpha}$ gives  
\begin{equation}
\begin{split}
\dim_{tH}X^{\alpha}&=\inf\{c: X^{\alpha}~\text{has basis}~\mathcal{U}^{\alpha},~\dim_H (\partial U^{\alpha})\leq c-1~\forall U^{\alpha}\in\mathcal{U}^{\alpha}\}\\
&=\inf\{c: X~\text{has basis}~\mathcal{U},~\frac{1}{\alpha}\dim_H \partial U\leq c-1~\forall U\in\mathcal{U}\}\\
&=\inf\{c: X~\text{has basis}~\mathcal{U},~\dim_H \partial U\leq (\alpha c-\alpha +1)-1~\forall U\in\mathcal{U}\}\\
\end{split}
\end{equation}
The desired equality now follows from the definition of tH-dimension.
\end{proof}

More interestingly, there exist spaces which are minimal for conformal dimension, yet their topological Hausdorff dimension can be lowered by quasisymmetric maps.  This is the content of Theorem~\ref{dimth theorem}.  Note that if $C$ is a compact Ahlfors regular  metric space, then $C\times [0,1]$ is minimal for conformal dimension~\cite{JT}, i.e. $\dim_C (C\times[0,1])=\dim_H(C\times[0,1]).$

\begin{theorem}\label{dimth theorem}
Let $C$ be an Ahlfors $d-$regular metric measure space with $0<d<1,$ and let $X=C\times [0,1].$  For every $\eps>0$ there is a metric space $Y$ and a quasisymmetric map $f:X\rightarrow Y$ such that $\dim_{tH}Y\leq1+\eps.$

As a consequence, $\dim_{tC}X=1.$
\end{theorem}

\begin{remark} If $C$ is the middle thirds Cantor set, under the assumptions of Theorem~\ref{dimth theorem} we have $\dim_{tH}X=1+d$ by Lemma~4.21~in~\cite{BBE2}, and $\dim_C X=1+d$ by~\cite{JT} and Corollary~8.11~in~\cite{PM}. 
\end{remark}

\begin{proof}
Let $D=\{\frac{a}{2^i}:a,i\in\mathbb{N},a\leq 2^i\}$ and write $D=\{y_1,y_2,\dots\}.$  For each $n$ let $\mu_n$ be the restriction of $\mathcal{H}^d$ to $C\times\{y_n\}.$
We define a measure $\nu$ on $X$ by 
\begin{align*}
\nu(E)=\sum_{n=1}^{\infty}\frac{\mu_n(E)}{2^n}
\end{align*}

Let $\alpha>0.$  By Theorem 4.1 in ~\cite{LK} there is a quasisymmetric map $f:X\rightarrow Y$ such that for all $z=(x,y)\in X$ and for all $r>0$
\begin{equation}
\begin{split}
f(B(z,r))\subset B(f(z),R),~~\text{where}~R=\min\{C(\alpha)r^{\nicefrac{1}{\alpha}}\nu(B(z,r))^{\nicefrac{-1}{\alpha}},r\}.\label{20}
\end{split}
\end{equation}
Since $C$ is Ahlfors $d-$regular, there is a constant $K$ such that for every $x\in C$ and for $0<r\leq\diam C$
\begin{equation}
\begin{split}
\nu(B((x,y_j),r))&=\sum_{n=1}^{\infty}\frac{\mu_n(B((x,y_j),r)}{2^n}\\
&\geq \frac{\mu_j(B((x,y_j),r)}{2^j}\\
&\geq K2^{-j} r^d\\
&=K_j r^d.\label{21}\\
\end{split}
\end{equation}
Let $x_1,x_2\in C$ and put $z_1=(x_1,y_j),~z_2=(x_2,y_j),~r=|x_1-x_2|=d_X(z_1,z_2)=|z_1-z_2|.$ By \eqref{21}
\begin{equation}
\begin{split}\label{22}
R&\leq \frac{C(\alpha)r^{\nicefrac{1}{\alpha}}}{\nu(B(z_1,r))^{\nicefrac{1}{\alpha}}}\\
&\leq \frac{C(\alpha)r^{\nicefrac{1}{\alpha}}}{K_j^{\nicefrac{1}{\alpha}}\left(r^d\right)^{\nicefrac{1}{\alpha}}}\\
&=C(j,\alpha)r^{\frac{1-d}{\alpha}}\\
&=C(j,\alpha)|z_1-z_2|^{\frac{1-d}{\alpha}}
\end{split}
\end{equation}
where $C(j,\alpha)$ is a constant that depends only on $j$ and $\alpha.$  Continuity of $f$ and \eqref{20} imply $f(\overline{B}(z_1,r))\subset\overline{B}(f(z_1),R).$ Then $|f(z_1)-f(z_2)|\leq R$, so by \eqref{22}
\begin{equation}\label{23}
|f(z_1)-f(z_2)|\leq R\leq C(j,\alpha)|z_1-z_2|^{\frac{1-d}{\alpha}}.
\end{equation}
By \eqref{23} we see that for all $j$
\begin{align*}
\dim_H f(C\times\{y_j\})\leq \frac{\alpha}{1-d}\dim_H (C\times\{y_j\})=\frac{\alpha d}{1-d}.
\end{align*}
Since $\alpha$ was arbitrary, for every $\eps>0$ there is a quasisymmetric map $f_{\eps}:X\rightarrow Y$ such that $\dim_H f_{\eps}(C\times\{y_j\})\leq \eps$ for all $j.$  Notice that since $f_{\eps}$ is a homeomorphism, $f_{\eps}(X)\setminus f_{\eps}(C\times D)$ is totally disconnected.  Also $\dim_H f_{\eps}(C\times D)=\sup_j\dim_H f_{\eps}(C\times \{y_j\})\leq \eps$ so that Theorem 3.7~in~\cite{BBE2} gives $\dim_{tH} f_{\eps}(X)\leq 1+\eps.$

To see that $\dim_{tC}X=1,$ first note that $\dim_{tC}X\geq 1$ by montonicity.  Quasisymmetric invariance of the tC-dimension coupled with \eqref{f1} shows that for every $\eps>0,$
\begin{equation*}
\dim_{tC}X=\dim_{tC}f(X)\leq\dim_{tH}f(X)\leq 1+\eps.  \qedhere
\end{equation*}
\end{proof}

As stated in the introduction, topological conformal dimension gives a lower bound on the topological Hausdorff dimensions of quasisymmetric images of a given metric space. 
\begin{equation}\label{24}
\dim_{tC}X\leq\inf\{\dim_{tH}f(X):f\text{ is quasisymmetric}\}.
\end{equation}

\begin{question}
Does equality hold in \eqref{24} for every metric space?
\end{question}

\subsection*{Acknowledgement} This paper is based on a part of a PhD thesis written by the author under the supervision of Leonid Kovalev.  The author thanks the anonymous referee for many useful suggestions in revising this paper.

\end{document}